\documentclass[12pt]{amsart}
\usepackage{ amsmath, amsthm, amsfonts, amssymb, color}
 \usepackage{mathrsfs}
\usepackage{amsfonts, amsmath}
 \usepackage{amsmath,amstext,amsthm,amssymb,amsxtra}
 \usepackage{txfonts} 
 \usepackage[colorlinks, citecolor=blue,pagebackref,hypertexnames=false]{hyperref}
 \allowdisplaybreaks
 \usepackage{pgf,tikz}
\begin{document}

 \baselineskip 16.6pt
\hfuzz=6pt

\widowpenalty=10000

\newtheorem{cl}{Claim}
\newtheorem{theorem}{Theorem}[section]
\newtheorem{proposition}[theorem]{Proposition}
\newtheorem{corollary}[theorem]{Corollary}
\newtheorem{lemma}[theorem]{Lemma}
\newtheorem{definition}[theorem]{Definition}
\newtheorem{assum}{Assumption}[section]
\newtheorem{example}[theorem]{Example}
\newtheorem{remark}[theorem]{Remark}
\renewcommand{\theequation}
{\thesection.\arabic{equation}}

\def\SL{\sqrt H}

\newcommand{\mar}[1]{{\marginpar{\sffamily{\scriptsize
        #1}}}}

\newcommand{\as}[1]{{\mar{AS:#1}}}

\newtheorem*{fact}{Fact}

\newtheorem{thm}{Theorem}[section]
\newtheorem{prop}[thm]{Proposition}
\newtheorem{cor}[thm]{Corollary}
\newtheorem{claim}{Claim}

\theoremstyle{definition}
\newtheorem{xca}[thm]{Exercise}

\newtheorem{question}[thm]{Question}

\newcommand{\cA}{\mathcal{A}}
\newcommand{\cB}{\mathcal{B}}
\newcommand{\cC}{\mathcal{C}}
\newcommand{\cF}{\mathcal{F}}
\newcommand{\cH}{{\bf H}}
\newcommand{\cK}{\mathcal{K}}
\newcommand{\cM}{\mathcal{M}}
\newcommand{\cN}{\mathcal{N}}
\newcommand{\cP}{\mathcal{P}}
\newcommand{\cS}{{\mathcal{S}}}
\newcommand{\cU}{\mathcal{U}}
\newcommand{\cX}{\mathcal{X}}
\newcommand{\dsone}{\mathds{1}}

\newcommand{\bB}{\mathbb{B}}
\newcommand{\bC}{\mathbb{C}}
\newcommand{\bD}{\mathbb{D}}
\newcommand{\bE}{\mathbb{E}}
\newcommand{\bF}{\mathbb{F}}
\newcommand{\bG}{\mathbb{G}}
\newcommand{\bN}{\mathbb{N}}
\newcommand{\bR}{\mathbb{R}}
\newcommand{\bZ}{\mathbb{Z}}
\newcommand{\bT}{\mathbb{T}}
\newcommand{\di}{\mathrm{div}}

\newcommand{\ctimes}{\rtimes_\theta}
\newcommand{\mtx}[4]{\left(\begin{array}{cc}#1&#2\\#3&#4\end{array} \right)}
\newcommand{\dvarphi}[2]{D^{(#1)}(\varphi_{#2})}
\newcommand{\xpsi}[2]{\psi^{[#1]}_{#2}(\omega)}
\newcommand{\xPsi}[2]{\Psi^{[#1]}_{#2}}
\newcommand{\dt}[1]{\partial_t^{(#1)}}
\newcommand{\vN}{vN(G)}
\newcommand{\add}[1]{\quad \text{ #1 } \quad}
\newcommand{\xp}{\rtimes_\theta }
\newcommand{\crossa}{\cM \rtimes_\theta G}
\newcommand{\uu}[2]{u_{#1}^{(#2)}}

\newcommand{\Lp}{{L_p(vN(G))}}
\newcommand{\Lplcr}[1]{{L_p(#1;\ell_2^{cr})}}
\newcommand{\Lplc}[1]{{L_p(#1;\ell_2^{c})}}
\newcommand{\Lplr}[1]{{L_p(#1;\ell_2^{r})}}
\newcommand{\Lpinfty}[1]{{L_p(#1;\ell_\infty)}}
\newcommand{\Lpli}[1]{{L_p(#1;\ell_1)}}
\newcommand{\Lpco}[1]{{L_p(#1;c_0)}}
\newcommand{\Lpcor}[1]{{L_p(#1;c_0)}}
\newcommand{\LpR}{{L_p(\bR^d;L_p(\cN))}}

\newcommand\R{\mathbb{R}}
\newcommand\RR{\mathbb{R}}
\newcommand\CC{\mathbb{C}}
\newcommand\NN{\mathbb{N}}
\newcommand\ZZ{\mathbb{Z}}
\newcommand\hDelta{{\bf L}}
\def\RN {\mathbb{R}^n}
\renewcommand\Re{\operatorname{Re}}
\renewcommand\Im{\operatorname{Im}}

\newcommand{\mc}{\mathcal}
\newcommand\D{\mathcal{D}}
\def\hs{\hspace{0.33cm}}
\newcommand{\la}{\alpha}
\def \l {\alpha}
\def\ls{\lesssim}
\def\su{{\sum_{i\in\nn}}}
\def\lz{\lambda}
\newcommand{\eps}{\varepsilon}
\newcommand{\pl}{\partial}
\newcommand{\supp}{{\rm supp}{\hspace{.05cm}}}
\newcommand{\x}{\times}
\newcommand{\lag}{\langle}
\newcommand{\rag}{\rangle}

\newcommand\wrt{\,{\rm d}}
\newcommand{\botimes}{\bar{\otimes}}
\def\nn{{\mathbb N}}
\def\bx{{\mathbb X}}
\def\fz{\infty}
\def\r{\right}
\def\lf{\left}
\def\cm{{\mathcal M}}
\def\cs{{\mathcal S}}
\def\rr{{\mathbb R}}
\def\lm{{L_{\infty}(\mathbb{R}^d)\overline{\otimes}\mathcal{M}}}
\def\rd{\mathbb{R}^d_{\theta}}
\def\ri{\mathbb{R}^2_{\theta}}
\def\rn{{{\rr}^d}}
\def\zz{{\mathbb Z}}
\def\cl{{\mathcal L}}
\def\cq{{\mathcal Q}}
\def\cd{{\mathcal D}}
\def\lz{{\lambda}}
\def\sz{{\sigma}}
\def\az{{\alpha}}
\def\dsup{{\sup}}

\theoremstyle{assumption}
\newtheorem{assumption}[theorem]{Assumption}

\title[Matrix-Weighted Besov Spaces Associated with Non-isotropic Dilations]{Matrix-Weighted Besov Spaces
Associated with Non-isotropic Dilations
}

%
%
%

\author{Xiong Liu}

\address{Xiong Liu, School of Mathematics and Physics\\
Gansu Center for Fundamental Research in Complex Systems Analysis and Control,
Lanzhou Jiaotong University, Lanzhou 730070\\ P. R. China}

\email{liuxmath@126.com}

\author{Wenhua Wang}
\address{Wenhua Wang (Corresponding author), Institute for Advanced Study in Mathematics\\Harbin Institute of Technology\\ Harbin 150001\\China}
\email{whwangmath@whu.edu.cn}


%

  \date{\today}
 \subjclass[2010]{42B25, 42B35, 47B38}
\keywords{Besov space, matrix weight, $\varphi$-transform, anisotropy, molecule.}

\begin{abstract}
Let $\alpha\in\mathbb{R}$, $p\in[1,\infty)$, $q\in(0,\infty]$, $\mathbf{W}$ be a matrix weight, and $A$ be an expansive dilation on $\mathbb{R}^d$.
  In this paper, the authors firstly investigate and develop some aspects of homogeneous anisotropic Besov spaces $\dot{B}^{\alpha,q}_{p,A}(\mathbb{R}^d,\mathbf{W})$ and inhomogeneous anisotropic Besov spaces $B^{\alpha,q}_{p,A}(\mathbb{R}^d,\mathbf{W})$
theory in the matrix weight setting. Moreover, we show that these spaces are characterized by the magnitude of the $\varphi$-transforms in appropriate sequence spaces. Notably, all these results remain novel even in the diagonal non-isotropic case (when $A = \mathrm{diag}(\lambda_1, \lambda_2, \ldots, \lambda_d)$ with  $\{\lambda_j\}_{j=1}^d \subset \mathbb{C}$).
\end{abstract}

\maketitle

 \tableofcontents


\section{Introduction}
\setcounter{equation}{0}
\subsection{Background and motivation}
As is known to all, the theory of Besov and Triebel-Lizorkin spaces has been developed and plays an essential role in the fields of harmonic analysis and partial differential equations (see e.g. \cite{c74,cw77,f07,fs72,g14,s60,t17}).
An important reason is that, such as Lebesgue, Hardy, Sobolev, Lipschitz spaces,
etc. are some special cases of either Besov spaces or Triebel-Lizorkin spaces
(see e.g. \cite{t83}). Moreover, these spaces can be characterized by their discrete
analogues: the sequence Besov spaces.

In 1997, Volberg presented a different solution
to the matrix weighted $L^p$ boundedness of the Hilbert transform via techniques
related to classical Littlewood-Paley theory (see e.g. \cite{v97}).
Soon
afterwards, in the matrix weight setting,
the theory of function spaces is one of the most fundamental and hot topics in the fields of harmoic
analysis and functional analysis. For examples, matrix $A_p$
weighted Besov spaces were introduced by Roudenko \cite{r03} and Frazier
and Roudenko \cite{fr04}. For more information about matrix-weighted function spaces, we refer the reader to see \cite{bcyy25,bhyy25,bhyy25x,bhyy24,bhyy,bhyy25xx,lw25}.

On the other hand,
 extending classic function spaces arising in harmonic analysis of Euclidean spaces $\rn$ to other domains and non-isotropic settings is a crucial and hot topic. For instance, in the 1970's,
 Calder\'{o}n and Torchinsky \cite{c77,ct75} studied the parabolic Hardy spaces and established their some real-variable characterizations.
 In 2003, Bownik \cite{b03} investigated the anisotropic Hardy space $H_A^p(\rn)$ with $p\in(0,\,\infty)$ and $A$ being a
 general expensive matrix on $\rn$, which was a generalization of parabolic Hardy spaces.
 In 2005, Bownik \cite{b05} developed anisotropic version of Besov spaces.
 In 2007, Bownik \cite{b07} also studied anisotropic Triebel-Lizorkin spaces with doubling measures.
For more information about anisotropic function spaces, we refer the reader to see \cite{b07,bh06,blyz10,lby14,lwyy19,lyy16}.

Inspired by the previous papers, the purpose of this paper is to extend some aspects of anisotropic function spaces theory, in particular, the study of Besov spaces, previously obtained with no weights and partially for scalar
weights, to the matrix weight setting.

\subsection{Organization}
This article is organized as follows.

In Section \ref{s2}, we first recall some notations and definitions
concerning anosotropic dilations and and matrix $A_p$ weight. Finally, we introduce the anisotropic Besov spaces in the convex of matrix weight.

The aim of Section \ref{s3} is  to establish the main results and give their proof.

In Section \ref{s4}, we will introduce the inhomogeneous anisotropic Besov spaces in the convex of matrix weight.

\subsection{Notation}
Finally, we make some conventions on notation.
Let $\nn:=\{0,\,1,\, 2,\,\ldots\}$ and $\zz_+:=\{1,\, 2,\,\ldots\}$.
For any $\alpha:=(\alpha_1,\ldots,\alpha_d)\in\zz_+^d:=(\zz_+)^d$, let
$|\alpha|:=\alpha_1+\cdots+\alpha_d$ and
$$\partial^\alpha:=
\lf(\frac{\partial}{\partial x_1}\r)^{\alpha_1}\cdots
\lf(\frac{\partial}{\partial x_d}\r)^{\alpha_d}.$$
Throughout the whole paper, we denote by $C$ a \emph{positive
constant} which is independent of the main parameters, but it may
vary from line to line. For any $p\in[1,\,\fz]$, let $p'$ denote the conjugate index of $p$, which satisfies that $\frac{1}{p}+\frac{1}{p'}=1$.
The \emph{symbol} $D\ls F$ means that $D\le
CF$.
 If $D\ls F$ and $F\ls D$, we then write $D\sim F$.
 Let ${\mathrm{I}}_{d\times d}$ denote the $d\times d$ unit matrix, and $\mathbf{0}$ the origin of $\rr^d$.
If the set $E \subset\rn$, let $E^\complement:=\rn\setminus E$ and we denote $\mathbf{1}_E$ its \emph{characteristic
function}.
 Denote by $\cs(\rn)$   the \emph{space of all Schwartz functions}.



\section{Preliminaries} \label{s2}
\setcounter{equation}{0}
 In this section, let us recall some basic definitions and properties of non-commutative $L_p$-spaces and anisotropic dilations, and introduce the omatrix-weighted anisotropic Besov  spaces.

\subsection{Anisotropic dilations on $\rn$}\label{n2.0}
Firstly, we recall the definition of {{expansive dilations}}
on $\rn$, which is introduced by Bownik (see \cite[p.\,5]{b03}). A real $d\times d$ matrix $A$ is called an {\it
expansive dilation} (shortly a {\it dilation}), if
$\min_{\lz\in\sz(A)}|\lz|>1$, where $\sz(A)$ denotes the set of
all {\it eigenvalues} of $A$. Let $\lz_-$ and $\lz_+$ be two {\it positive numbers} such that
$$1<\lz_-<\min\{|\lz|:\ \lz\in\sz(A)\}\le\max\{|\lz|:\
\lz\in\sz(A)\}<\lz_+.$$
From \cite[Lemma 2.2]{b03}, we know that, for a fixed dilation $A$,
there exist a number $r\in(1,\,\fz)$ and a set $\Delta:=\{x\in\rn:\,|Px|<1\}$, where $P$ is some non-degenerate $d\times d$ matrix, such that $\Delta\subset r\Delta\subset A\Delta,$ and one can additionally
assume that $|\Delta|=1$, where $|\Delta|$ denotes the
$d$-dimensional Lebesgue measure of the set $\Delta$. For $k\in \zz$, let
$B_k:=A^k\Delta$. Then $B_k$ is open, and $$B_k\subset
rB_k\subset B_{k+1}, \ \ |B_k|=|\det A|^k.$$
An ellipsoid $x+B_k$ for some $x\in\rn$ and $k\in\zz$ is called a {\it dilated ball}.
Define
\begin{eqnarray}\label{e2.1}
\mathfrak{B}:=\lf\{x+B_k:\ x\in \rn,\,k\in\zz\r\}.
\end{eqnarray}

Throughout the whole paper, let $\sigma$ be the {\it smallest integer} such that $2B_0\subset A^\sigma B_0$. Then,
for all $k,\,j\in\zz$ with $k\le j$, it holds true that
\begin{eqnarray}
&&B_k+B_j\subset B_{j+\sz},\label{e2.3}\\
&&B_k+(B_{k+\sz})^\complement\subset
(B_k)^\complement,\label{e2.4}
\end{eqnarray}
where $E+F$ denotes the {\it algebraic sum} $\{x+y:\ x\in E,\,y\in F\}$
of  sets $E,\, F\subset \rn$.

\begin{definition}
 A \textit{quasi-norm}, associated with
dilation $A$, is a Borel measurable mapping
$\rho_{A}:\rr^{n}\to [0,\infty)$, for simplicity, denoted by
$\rho$, satisfying
\begin{enumerate}
\item[\rm{(i)}] $\rho(x)>0$ for all $x \in \rn\setminus\{ \mathbf{0}\}$,
here and hereafter, $\mathbf{0}$ denotes the origin of $\rn$;
\item[\rm{(ii)}] $\rho(Ax)= |\det A|\rho(x)$ for any $x\in \rr^{d}$;
\item[\rm{(iii)}] $ \rho(x+y)\le C_A\lf[\rho(x)+\rho(y)\r]$ for
all $x,\, y\in \rr^{d}$, where $C_A\in[1,\,\fz)$ is a constant independent of $x$ and $y$.
\end{enumerate}
\end{definition}
\begin{remark}{\rm
\begin{enumerate}
\item[\rm{(i)}]
In the standard dyadic case $A:=2{{\rm I}_{d\times d}}$, $\rho(x):=|x|^d$ for
all $x\in\rn$ is
an example of homogeneous quasi-norms associated with $A$, here and hereafter, ${\rm I}_{d\times d}$ denotes the $d\times d$ {\it unit matrix},
$|\cdot|$ always denotes the {\it Euclidean norm} in $\rn$.
\item[\rm{(ii)}]
From \cite[Lemma 3.2]{b03}, we find that, there exists a positive constant $C$ such that, for all $x\in\rn$,
\begin{align}\label{e2.3x}
&\frac{1}{C}[\rho(x)]^{\zeta_-}\leq|x|\leq C[\rho(x)]^{\zeta_+},\ \ \mathrm{if}\ \ \rho(x)\geq1,\\ \nonumber
&\frac{1}{C}[\rho(x)]^{\zeta_+}\leq|x|\leq C[\rho(x)]^{\zeta_-},\ \ \mathrm{if}\ \ \rho(x)<1,\nonumber
\end{align}
where $\zeta_+:=\frac{\ln\lambda_+}{\ln |\det A|}$ and $\zeta_-:=\frac{\ln\lambda_-}{\ln |\det A|}$.
\item[\rm{(iii)}] Bownik has obtained that, all homogeneous quasi-norms associated with a given dilation
$A$ are equivalent, see \cite[Lemma 2.4]{b03}. Therefore, in what follows, for a
given dilation $A$, for simplicity, we
always use the {\it{step homogeneous quasi-norm}} $\rho$ defined by setting,  for all $x\in\rn$,
\begin{equation*}
\rho(x):=\sum_{k\in\zz}|\det A|^k\mathbf{1}_{B_{k+1}\setminus B_k}(x)\ {\rm
if} \ x\ne \mathbf{0},\hs {\mathrm {or\ else}\hs } \rho(\mathbf{0}):=0.
\end{equation*}
By \eqref{e2.3}, we can deduce that, for all $x,\,y\in\rn$,
$$\rho(x+y)\le
|\det A|^\sz\lf(\max\lf\{\rho(x),\,\rho(y)\r\}\r)\le |\det A|^\sz[\rho(x)+\rho(y)].$$
\end{enumerate}}
\end{remark}

\subsection{Matrix $A_p$ weight }\label{n2.0}
In this subsection, we will recall the definitions of  anisotropic matrix $A_p$ weight.

Let $\cm$ be the cone of nonnegative
definite operators on a Hilbert space $\mathcal{H}$ of dimension $n$ (consider $H=\mathbb{C}^n$ or
$\rr^n$), i.e., for $M\in\cm$ we have $\langle Mx,x\rangle_H \geq0$, for all $x\in\mathcal{H}$. Therefore, a
matrix weight $\mathbf{W}$ is an a.e. invertible map $\mathbf{W}:\rr^d\rightarrow\cm$.
Define the matrix-weighted Lebesgue space $L^p(\rr^d,\mathbf{W})$ as all measurable vector-valued function $\vec{f}:=(f_1,f_2,\ldots,f_m)^T:\rr^d\rightarrow \mathcal{H}$ satisfy
$$\lf\|\vec{f}\,\r\|_{L^p(\rr^d,\mathbf{W})}:=
\lf(\int_{\rr^d}\lf\|\mathbf{W}^{\frac1p}(x)\vec{f}(x)\r\|_{\mathcal{H}}\,dx\r)^{\frac1p}<\fz.$$

Since the $A_p$ condition can be formulated for any family of norms $\omega_t$ on a
Hilbert space $\mathcal{H}$, we will recall the following norms:
$$\omega_t(y)=\lf\|\mathbf{W}^{\frac1p}(x)y\r\|_{\mathcal{H}}, \ \mathrm{where}\  y\in\mathcal{H},
x\in\rr^d.$$
Its dual metric is defined by
$$\omega_t^*(x)=\sup_{z\in\mathcal{H},z\neq0}\frac{\langle x,z\rangle_{\mathcal{H}}}{\omega_t(z)}=\lf\|\mathbf{W}^{\frac1p}(x)y\r\|_{\mathcal{H}}, \ \mathrm{where}\  y\in\mathcal{H},
x\in\rr^d.$$
Define the norm $\omega_{p,B}$ through the averagings of the metrics
$\omega_t$ on the ball $B\in\mathfrak{B}$:
$$\omega_{p,B}(x)=\lf[\fint_B\lf(\omega_t(x)\r)^p\,dt\r]^{\frac1p}.$$
Its dual metric is defined by
$$\omega^*_{p',B}(x)=\lf[\fint_B\lf(\omega_t^*(x)\r)^{p'}\,dt\r]^{\frac{1}{p'}}.$$

The following is the definition of $A_p$ matrix weight, which is from \cite{}.
\begin{definition}
Let $p\in(1,\fz)$. We say that $\mathbf{W}$ is an $A_p$ matrix weight if
$\mathbf{W}:\rr^d\rightarrow\cm$ is such that
$\mathbf{W}$ and $\mathbf{W}^{-\frac{p'}{p}}$ are locally integrable and there exists $C<0$ such that
$$\omega^*_{p',B}\leq C(\omega_{p,B})^*.$$
\end{definition}

\begin{definition}
Let $p\in[1,\fz)$. We say that matrix weight $\mathbf{W}$ is a doubling matrix of order $p$, if
there exists $C_{p,d}$ such that, for any $y\in\mathcal{H}$, and $B\in\mathfrak{B}$,
$$\int_{2B}\lf\|\mathbf{W}^{\frac1p}(x)y\r\|_{\mathcal{H}}^p\,dx
\leq C_{p,d}\int_{B}\lf\|\mathbf{W}^{\frac1p}(x)y\r\|_{\mathcal{H}}^p\,dx.$$

\end{definition}

The following are some equivalent definitions of matrix $A_p$ weight in the anisotropic setting, whose proof is similar to Roudenko's ones \cite{r03}.

\begin{lemma}
The following condition are equivalllent:
\begin{enumerate}
\item[\rm{(i)}] $\mathbf{W}\in A_p$;

\item[\rm{(ii)}]
$\mathbf{W}^{-\frac{p'}{p}}\in A_p$;

\item[\rm{(iii)}]
$\fint_B\lf[\fint_B\lf\|\mathbf{W}^{\frac1p}(x)\mathbf{W}^{-\frac1p}(y)\r\|^{p'}dy\r]
^{\frac{p}{p'}}dx\leq C$, \ \ for any  $B\in\mathfrak{B}$;

\item[\rm{(iii)}]
$\fint_B\lf[\fint_B\lf\|\mathbf{W}^{\frac1p}(x)\mathbf{W}^{-\frac1p}(y)\r\|^{p}dy\r]
^{\frac{p'}{p}}dx\leq C$, \ \ for any  $B\in\mathfrak{B}$,
\end{enumerate}
where $\|\mathbf{W}^{\frac1p}(x)\mathbf{W}^{-\frac1p}(y)\|$ denotes the matrix norm.
\end{lemma}

\subsection{Matrix-weighted Besov spaces associated with anisotropic dilations}\label{n2.0}
In this subsection, we will introduce the definitions of anisotropic matrix-weighted Besov spaces.

A $C^\infty(\rn)$ function $\varphi$ is said to belong to the Schwartz class $\cs(\rn)$ if,
for any $\ell\in\zz_+$ and multi-index $\alpha$,
$\|\varphi\|_{\alpha,\ell}:=\dsup_{x\in\rn}[\rho(x)]^\ell|\partial^\az\varphi(x)|<\infty$.

Let $\varphi\in\cs(\rr^d)$ and satisfy
\begin{align*}
\supp \hat{\varphi}\subset[-\pi,\pi]^d\backslash\{\mathbf{0}\}
\end{align*}
and
\begin{align*}
\supp \lf|\hat{\varphi}((A^*)^{-k}\xi)\r|>0,  \ \ \mathrm{for \ any} \ \xi\in\rr^d\backslash\{\mathbf{0}\}.
\end{align*}

In what follows, for $\varphi\in \cs(\rn)$, $k\in\zz$ and $x\in\rn$, let $$\varphi_k(x):= |\det A|^{-k}\varphi(A^{-k}x).$$
\begin{definition}
Let $\alpha\in\rr$, $p\in[1,\fz)$, $q\in(0,\fz]$ and $\mathbf{W}$ be a matrix weight.
The Besov space $\dot{B}_{p,A}^{\alpha,q}(\rr^d,\mathbf{W})$ is defined by, all vector-valued distributions
$\vec{f}:=(f_1,f_2,\ldots,f_m)^T$ with $f_i\in\cs'/\mathcal{P}(\rr^d)$, $i=1,\ldots,m$, such that
$$\|f\|_{\dot{B}_{p,A}^{\alpha,q}(\rr^d,\mathbf{W})}:=
\lf(\sum_{k\in\zz}|\det A|^{k\alpha q}\lf\|\vec{f}*\varphi_k\r\|_{L^p(\mathbf{W})}^q\r)^{\frac1q},$$
where $\vec{f}*\varphi_k:=(f_1*\varphi_k,f_2*\varphi_k,\ldots,f_m*\varphi_k)^T$.
\end{definition}

Let $\varphi\in\cs$ and satisfy
\begin{align*}
\supp \hat{\varphi}, \supp \hat{\psi}\subset[-\pi,\pi]^d\backslash\{\mathbf{0}\}
\end{align*}
and
\begin{align}\label{test2}
\sum_{k\in\zz}\overline{\hat{\varphi}((A^*)^{-k}\xi)}
\hat{\psi}((A^*)^{-k}\xi)=1,  \ \ \mathrm{for \ any} \ \xi\in\rr^d\backslash\{\mathbf{0}\}.
\end{align}
For any $\mathbf{j}\in\zz^d$, $k\in\zz$, define
$$Q_{\mathbf{j},k}=A^{-k}([0,1]^d+\mathbf{j})$$
be the dilated cube, $x_{Q_{\mathbf{j},k}}=A^{-k}\mathbf{j}$.
$$\mathcal{Q}=\lf\{Q_{\mathbf{j},k}:k\in\zz, \mathbf{j}\in\zz^d\r\}.$$
Define
$\varphi_Q=|\det A|^{-\frac{k}{2}}\varphi(A^{k}x-\mathbf{j})=
|Q|^{\frac12}\varphi(x-x_Q)$.

For any $\vec{f}$, $f_i\in\cs'$, define $\varphi$-transform $S_{\varphi}$:
$$S_{\varphi}(\vec{f})=\lf\{\langle\vec{f},\varphi_Q\rangle\r\}_Q
=\lf\{\lf(\langle f_1,\varphi_Q\rangle,\langle f_2,\varphi_Q\rangle,\ldots,\langle f_m,\varphi_Q\rangle\r)^T\r\}_Q.$$
Moreover, we call $\vec{s}_Q(\vec{f})=\langle \vec{f},\varphi_Q\rangle$ the $\varphi$-transform
coefficients of $\vec{f}$.

 Similarly to $\varphi_Q$, define
$$\psi_Q=|\det A|^{-\frac{k}{2}}\psi(A^{k}x-\mathbf{j})=
|Q|^{\frac12}\psi(x-x_Q).$$
The inverse $\varphi$-transform $T_{\psi}$ is the map taking a sequence $s=\{s_Q\}$ to
$T_{\psi}s=\sum_{Q\in\mathcal{Q}}s_Q\psi_Q$. In the vector case,
$$T_{\psi}\vec{s}=\sum_{Q\in\mathcal{Q}}\vec{s}_Q\psi_Q,$$
where $\vec{s}_Q\psi_Q=(s_Q^1\psi_Q,s_Q^2\psi_Q,\ldots,s_Q^m\psi_Q)^T$. The $\varphi$-transform
decomposition show that, for any $f\in\cs'\backslash\mathcal{P}(\rr^d)$,
$$f=\sum_{Q}\langle f,\varphi_Q\rangle\psi_Q=:\sum_{Q}s_Q\psi_Q.$$

\begin{definition}
Let $\alpha\in\rr$, $p\in[1,\fz)$, $q\in(0,\fz]$ and $\mathbf{W}$ be a matrix weight.
The Besov space $\dot{b}_{p,A}^{\alpha,q}(\rr^d,\mathbf{W})$ is defined by, all vector-valued sequence
$\vec{s}:=\{\vec{s}_Q\}_{Q\in\mathcal{Q}}$, where $\vec{s}_Q=(s_Q^1,s_Q^2,\ldots,s_Q^m)^T$, $i=1,\ldots,m$, such that
$$\|\vec{s}\|_{\dot{b}_{p,A}^{\alpha,q}(\mathbf{W})}:=
\lf(\sum_{k\in\zz}|\det A|^{k\alpha q}\lf\|\sum_{Q\in\mathcal{Q},|Q|=|\det A|^{-k}}|Q|^{-\frac12}|\vec{s}_Q|\mathbf{1}_Q\r\|_{L^p(\mathbf{W})}^q\r)^{\frac1q}.$$
\end{definition}

\begin{remark}{\rm
\begin{enumerate}
\item[\rm{(i)}]
In the isotropic setting, i.e.,
 when $A=2\mathrm{I}_{d\times d}$, these spaces reduce to the weighted-matrix Besov space $\dot{B}^{\alpha,q}_{p}(\rr^d,\mathbf{W})$.
\end{enumerate}}
\end{remark}

\section{The main results and their proof}\label{s3}
\subsection{The boundedness of inverse $\varphi$-transform}

\begin{definition}
Let $\alpha\in\rr$, $p,q\in(0,\fz]$, $\delta\in(\alpha-\lfloor\alpha\rfloor,1]$. Set $J=\frac{\beta}{p}+\max\{0,1-\frac{1}{p}\}$,
$N=\max\{\lfloor\frac{J-\alpha-1}{\zeta_-}\rfloor,-1\}$
We say that $\mathfrak{M}_Q$ is a smooth $(M,N,\delta)$-molecule associated with dyadic cube $Q\in\mathcal{Q}$ with $|Q|=|\det A|^{-k},k\in\zz$ if there exists $M>J$ such that
\begin{enumerate}
\item[\rm{(i)}]
$\lf|\mathfrak{M}_Q(x)\r|\leq|\det A|^{\frac k2}\frac{1}{\lf[1+\rho(A^k(x-x_Q))\r]^{\max\lf\{M,\frac{(M-\alpha)\zeta_+}{\zeta_-}\r\}}};$

\item[\rm{(ii)}]
$\int_{rr^d}x^{\gamma}\mathfrak{M}_Q(x)\,dx=0$, \ for $|\gamma|\leq N$;

\item[\rm{(iii)}]
$\lf|\partial^{\gamma}\mathfrak{M}_Q(A^{-k}\cdot)(x)\r|\leq|\det A|^{\frac k2}\frac{1}{[1+\rho(x-A^kx_Q)]^M},$ \ for $|\gamma|\leq \lfloor\frac{\alpha}{\zeta_-}\rfloor+1$;

\item[\rm{(iv)}]
$\lf|\partial^{\gamma}\mathfrak{M}_Q(A^{-k}\cdot)(x)-
\partial^{\gamma}\mathfrak{M}_Q(A^{-k}\cdot)(x)\r|\leq|\det A|^{\frac k2-\gamma-\delta}\rho(x-y)^{\delta}\sup_{\rho(z)\leq\rho(x-y)}\frac{1}{\lf[1+\rho(x-z-A^kx_Q)\r]^M}$,
\ \ \ if $|\gamma|=\lfloor\frac{\alpha}{\zeta_-}\rfloor$.

\end{enumerate}

\end{definition}
The proof of the following lemma is similar to \cite{r03}.
\begin{lemma}\label{mole}
Let $\alpha\in\rr$, $p\in[1,\fz)$, $q\in(0,\fz]$, and $\mathbf{W}$ be a doubling matrix weight of order $p$. Assume that $\{\mathfrak{M}_Q\}_Q$ is a family of smooth molecules for $\dot{B}_{p,A}^{\alpha,q}(\rr^d,\mathbf{W})$, then we have
$$\lf\|\sum_{Q\in\mathcal{Q}}\vec{s}_Q\mathfrak{M}_Q\r\|_{\dot{B}_{p,A}^{\alpha,q}(\rr^d,\mathbf{W})}\leq C\lf\|\{\vec{s}_Q\}_{Q\in\mathcal{Q}}\r\|_{\dot{b}_{p,A}^{\alpha,q}}.$$
\end{lemma}
%

\begin{thm}
Let $\alpha\in\rr$, $p\in[1,\fz)$, $q\in(0,\fz]$, and $\mathbf{W}$ be a doubling matrix weight of order $p$. Then, for any sequence $\vec{s}=\{\vec{s}_Q\}_Q\in \dot{b}^{\alpha,q}_{p,A}(\mathbf{W})$, we have
$$\lf\|T_{\psi}\vec{s}\r\|_{\dot{B}_{p,A}^{\alpha,q}(\rr^d,\mathbf{W})}=\lf\|\sum_{Q\in\mathcal{Q}}
\vec{s}_Q\psi_Q\r\|_{\dot{B}_{p,A}^{\alpha,q}(\rr^d,\mathbf{W})}\leq C\lf\|\{\vec{s}_Q\}_{Q\in\mathcal{Q}}\r\|_{\dot{b}_{p,A}^{\alpha,q}(\mathbf{W})}.$$
\end{thm}
\begin{proof}
Its proof can be deduced by the fact that $\{\psi_Q\}_{Q\in\mathcal{Q}}$ is a family of smooth molecules for
$\dot{B}_{p,A}^{\alpha,q}(\rr^d,\mathbf{W})$.

\end{proof}

\subsection{The boundedness of $\varphi$-transform}

\begin{definition}
For $k\in\zz$, define
$$E_k:=\lf\{\vec{f}\in L^p(\mathbf{W}):\supp \hat{f}_i\subseteq B_{k+1}, i=1,2,\ldots,m\r\},$$
then we say that $E_k$ consists of vector function of exponential type $|\det A|^{k+1}$.
\end{definition}

\begin{lemma}
Let $f\in\cs'$, $g\in\cs(\rr^d)$, and $\supp \hat{f}$, $\supp \hat{g}\subseteq (A^*)^k[-\pi,\pi]^d$ with $k\in\zz$. Then we have
$$f*g(x)=\sum_{\mathbf{j}\in\zz^d}|\det A|^{-k}f(A^{-k}\mathbf{j})g(x-A^{-k}\mathbf{j}).$$

\end{lemma}

The following is maximal operator type inequality.
\begin{lemma}
Let $p\in(1,\fz)$, $\mathbf{W}\in A_p$ and $\vec{f}\in E_0$. Then we have
$$\sum_{\mathbf{j}\in\zz^d}\int_{Q_{\mathbf{j},0}}\lf\|\mathbf{W}^{\frac1p}(x)\vec{f}(\mathbf{j})\r\|^pdx
\leq C_{p,d}\lf\|\vec{f}\,\r\|_{L^p(\mathbf{W})}.$$
\end{lemma}
\begin{proof}
Choose a smooth function $\phi\in\cs$ satisfying:
if $\xi\in B_1$, $\hat{\phi}=1$, and $\supp \hat{\phi}\subseteq[-\pi,\pi]^d$.
Then for any $\vec{f}\in E_0$, we have $\vec{f}=\vec{f}*\phi$.
Therefore,
 \begin{align*}
 \sum_{\mathbf{j}\in\zz^d}\int_{Q_{\mathbf{j},0}}\lf\|\mathbf{W}^{\frac1p}(x)\vec{f}(\mathbf{j})\r\|^pdx
=&\sum_{\mathbf{j}\in\zz^d}\int_{Q_{\mathbf{j},0}}\lf\|\mathbf{W}^{\frac1p}(x)\int_{\rr^d}\vec{f}(y)\phi(\mathbf{j}-y)\,dy\r\|
^pdx\\
\lesssim&\sum_{\mathbf{j}\in\zz^d}\int_{Q_{\mathbf{j},0}}\lf[\int_{\rr^d}\lf\|
\mathbf{W}^{\frac1p}(x)\vec{f}(y)\frac{1}{[\rho(1+(\mathbf{j}-y))]^N}\r\|\,dy\r]^p
dx,
\end{align*}
where $N>d+\frac{\beta p}{p'}$, $\beta$ is the doubling exponent of matrix weight $\mathbf{W}$.

By the discrete H\"{o}lder inequality, we obtain
\begin{align*}
&\sum_{\mathbf{j}\in\zz^d}\int_{Q_{\mathbf{j},0}}\lf[\int_{\rr^d}\lf\|
\mathbf{W}^{\frac1p}(x)\vec{f}(y)\frac{1}{[\rho(1+(\mathbf{j}-y))]^N}\r\|\,dy\r]^p
dx\\
\lesssim&\sum_{\mathbf{j}\in\zz^d}\int_{Q_{\mathbf{j},0}}\lf[\sum_{\mathbf{j}'\in\zz^d}
\int_{Q_{\mathbf{j}',0}}\lf\|
\mathbf{W}^{\frac1p}(x)\vec{f}(y)\frac{1}{[\rho(1+(\mathbf{j}-y))]^N}\r\|\,dy\r]^p
dx\\
\lesssim&\sum_{\mathbf{j}\in\zz^d}\int_{Q_{\mathbf{j},0}}\sum_{\mathbf{j}'\in\zz^d}
\lf[\int_{Q_{\mathbf{j}',0}}\lf\|
\mathbf{W}^{\frac1p}(x)\vec{f}(y)\r\|\,dy\r]^p\frac{1}{[\rho(1+(\mathbf{j}-y))]^N}\cdot
\lf(\sum_{\mathbf{j}'\in\zz^d}\frac{1}{[\rho(1+(\mathbf{j}-y))]^N}\r)^{\frac{p}{p'}}
dx\\
\lesssim&\sum_{\mathbf{j}\in\zz^d}\int_{Q_{\mathbf{j},0}}\sum_{\mathbf{j}'\in\zz^d}
\lf[\int_{Q_{\mathbf{j}',0}}\lf\|
\mathbf{W}^{\frac1p}(x)\vec{f}(y)\r\|\,dy\r]^p\frac{1}{[\rho(1+(\mathbf{j}-y))]^N}
dx.
\end{align*}
It follows from the H\"{o}lder inequality that
\begin{align*}
\lf[\int_{Q_{\mathbf{j}',0}}\lf\|
\mathbf{W}^{\frac1p}(x)\vec{f}(y)\r\|\,dy\r]^p
&\leq\lf[\int_{Q_{\mathbf{j}',0}}\lf\|\mathbf{W}^{\frac1p}(x)\mathbf{W}^{-\frac1p}(y)\r\|\lf\|
\mathbf{W}^{\frac1p}(y)\vec{f}(y)\r\|\,dy\r]^p\\
&\leq\lf[\int_{Q_{\mathbf{j}',0}}\lf\|\mathbf{W}^{\frac1p}(x)\mathbf{W}^{-\frac1p}(y)\r\|^{p'}\,dy\r]^
{\frac{p}{p'}}
\cdot\int_{Q_{\mathbf{j}',0}}\lf\|
\mathbf{W}^{\frac1p}(y)\vec{f}(y)\r\|^p\,dy.
\end{align*}
Now we can prove that
$$\sum_{\mathbf{j}\in\zz^d}\int_{Q_{\mathbf{j},0}}\sum_{\mathbf{j}'\in\zz^d}
\lf[\int_{Q_{\mathbf{j}',0}}\lf\|\mathbf{W}^{\frac1p}(x)\mathbf{W}^
{-\frac1p}(y)\r\|^{p'}\,dy\r]^
{\frac{p}{p'}}\frac{1}{[\rho(1+(\mathbf{j}-y))]^N}
dx\lesssim1.$$
Therefore, we have
\begin{align*}
 \sum_{\mathbf{j}\in\zz^d}\int_{Q_{\mathbf{j},0}}\lf\|\mathbf{W}^{\frac1p}(x)\vec{f}(\mathbf{j})\r\|^pdx
 &\lesssim\sum_{\mathbf{j}'\in\zz^d}\int_{Q_{\mathbf{j}',0}}\lf\|
\mathbf{W}^{\frac1p}(y)\vec{f}(y)\r\|^p\,dy\\
&\thicksim\int_{\rr^d}\lf\|
\mathbf{W}^{\frac1p}(y)\vec{f}(y)\r\|^p\,dy\\
&\thicksim\lf\|\vec{f}\,\r\|_{L^p(\mathbf{W})}^p.
\end{align*}

\end{proof}
\begin{lemma}
Let $\mathbf{W}$ be a doubling matrix of order $p$ with doubling exponent $\beta<p$, and $\vec{f}\in E_0$. Then we have
$$\sum_{\mathbf{j}\in\zz^d}\int_{Q_{\mathbf{j},0}}\lf\|\mathbf{W}^{\frac1p}(x)\vec{f}(\mathbf{j})\r\|^pdx
\leq C_{p,d}\lf\|\vec{f}\,\r\|_{L^p(\mathbf{W})}.$$
\end{lemma}

\begin{thm}\label{t3.3}
Let $\alpha\in\rr$, $q\in(0,\fz]$, $p\in[1,\fz)$, $\mathbf{W}\in A_p$.
Then for any $f\in\dot{B}_{p,A}^{\alpha,q}(\rr^d,\mathbf{W})$, we have
$$\lf\|\{\vec{s}_Q\}_Q\r\|_{\dot{b}_{p,A}^{\alpha,q}(\mathbf{W})}\lesssim
\lf\|\vec{f}\,\r\|_{\dot{B}_{p,A}^{\alpha,q}(\rr^d,\mathbf{W})},$$
where $\vec{s}_Q=S_{\varphi}(\vec{f})=\langle\vec{f},\varphi_Q\rangle
=\lf(\langle f_1,\varphi_Q\rangle,\langle f_2,\varphi_Q\rangle,\ldots,\langle f_m,\varphi_Q\rangle\r)^T.$
\end{thm}

\begin{proof}
By the definition, we have
\begin{align*}
\lf\|\{\vec{s}_Q\}_Q\r\|_{\dot{b}_{p,A}^{\alpha,q}(\mathbf{W})}
=&\lf(\sum_{k\in\zz}|\det A|^{k\alpha q}\lf\|\sum_{Q\in\mathcal{Q},|Q|=|\det A|^{-k}}|Q|^{-\frac12}|\vec{s}_Q|\mathbf{1}_Q\r\|_{L^p(\mathbf{W})}^q\r)^{\frac1q}\\
=&\lf(\sum_{k\in\zz}|\det A|^{k\alpha q}\lf\|\sum_{Q\in\mathcal{Q},|Q|=|\det A|^{-k}}|Q|^{-\frac12}\lf\|\mathbf{W}^{\frac1p}(x)\vec{s}_Q\r\|\mathbf{1}_Q(x)\r\|_{L^p(dx)}^q\r)
^{\frac1q}.
\end{align*}
Fix $k\in\zz$, let $Q_{\mathbf{j},k}=A^{-k}([0,1]^d+\mathbf{j})$, $\mathbf{j}\in\zz^d$. Then we have $$\vec{s}_Q=\lf\langle\vec{f},\varphi_Q\r\rangle=|Q|^{\frac12}\lf(\vec{f}*\widetilde{\varphi}_k\r)
(A^{-k}\mathbf{j})$$ and
\begin{align*}
\lf\|\sum_{Q\in\mathcal{Q},|Q||\det A|^{-k}}|Q|^{-\frac p2}\lf\|\mathbf{W}^{\frac1p}(x)\vec{s}_Q\r\|\mathbf{1}_Q(x)\r\|_{L^p(dx)}^p
=&\sum_{Q\in\mathcal{Q},|Q|=|\det A|^{-k}}|Q|^{-\frac p2}\int_Q\lf\|\mathbf{W}^{\frac1p}(x)|Q|^{\frac12}(\vec{f}*\widetilde{\varphi}_k)(A^{-k}\mathbf{j})
\r\|^p\,dx\\
=&\sum_{\mathbf{j}\in\zz^d}\int_{Q_{\mathbf{j},k}}\lf\|\mathbf{W}^{\frac1p}(x)(\vec{f}*
\widetilde{\varphi}_k)(A^{-k}\mathbf{j})\r\|^p\,dx.
\end{align*}
Set $\vec{f}_k(x):=\vec{f}(A^{-k}x)$ and $\mathbf{W}_k(x):=\mathbf{W}(A^kx)$, then $\vec{f}*\widetilde{\varphi}_k(A^{-k}\mathbf{j})=\vec{f}_k*\widetilde{\varphi}(\mathbf{j})$. Therefore,
\begin{align*}
\lf\|\sum_{Q\in\mathcal{Q},|Q|=|\det A|^{-k}}|Q|^{-\frac p2}\lf\|\mathbf{W}^{\frac1p}(x)\vec{s}_Q\r\|\mathbf{1}_Q(x)\r\|_{L^p(dx)}^p
=&|\det A|^{-k}\sum_{\mathbf{j}\in\zz^d}\int_{Q_{\mathbf{j},0}}\lf\|\mathbf{W}_k^{\frac1p}(x)(\vec{f}_k*
\widetilde{\varphi})(\mathbf{j})\r\|^p\,dx\\
\lesssim&|\det A|^{-k}\int_{\rr^d}\lf\|\mathbf{W}^{\frac1p}(x)(\vec{f}*
\widetilde{\varphi}_k)(x)\r\|^p\,dx\\
\thicksim&\int_{\rr^d}\lf\|\mathbf{W}_k^{\frac1p}(x)(\vec{f}_k*
\widetilde{\varphi}_k)(x)\r\|^p\,dx\\
\thicksim&\lf\|\vec{f}_k*
\widetilde{\varphi}_k\r\|_{L^p(\mathbf{W})}^p.
\end{align*}
From the above estimates, we conclude that
\begin{align*}
\lf\|\{\vec{s}_Q\}_Q\r\|_{\dot{b}_{p,A}^{\alpha,q}(\mathbf{W})}
\lesssim&\lf(\sum_{k\in\zz}|\det A|^{k\alpha q}\lf\|\vec{f}_k*
\widetilde{\varphi}_k\r\|_{L^p(\mathbf{W})}^q\r)^{\frac1q}\\
\thicksim&\|f\|_{\dot{B}_{p,A}^{\alpha,q}(\rr^d,\mathbf{W})}.
\end{align*}
Finally, we only need to show the equivalence of  $\dot{B}_{p,A}^{\alpha,q}(\rr^d,\mathbf{W},\varphi)=\dot{B}_{p,A}^{\alpha,q}(\rr^d,\mathbf{W},\widetilde{\varphi})$.
It is easy to see that $\hat{\widetilde{\varphi}}=\overline{\hat{\varphi}}$ and $\hat{\widetilde{\psi}}=\overline{\hat{\psi}}$. Therefore, we know that
$(\widetilde{\varphi},\widetilde{\psi})$ satisfies \eqref{test2}.
From the decomposition of
$$\vec{f}=\sum_{Q\in\mathcal{Q}}\lf\langle \vec{f},\widetilde{\varphi}_Q\r\rangle\widetilde{\psi}_Q,$$
the fact that $\{\widetilde{\psi}_Q\}_{Q\in\mathcal{Q}}$ is a family of smooth molecules for $\dot{B}_{p,A}^{\alpha,q}(\rr^d,\mathbf{W})$, and Lemma \ref{mole}, we conclude that
$$\lf\|\vec{f}\,\r\|_{\dot{B}_{p,A}^{\alpha,q}(\rr^d,\mathbf{W},\widetilde{\varphi})}
\lesssim\lf\|\lf\{\lf\langle \vec{f},\widetilde{\varphi}_Q\r\rangle\r\}_{Q\in\mathcal{Q}}\r\|_
{\dot{b}_{p,A}^{\alpha,q}(\mathbf{W})}.$$
By this and the above estimate, we obtain
$$\lf\|\vec{f}\,\r\|_{\dot{B}_{p,A}^{\alpha,q}(\rr^d,\mathbf{W},\widetilde{\varphi})}
\lesssim\lf\|\vec{f}\,\r\|_{\dot{B}_{p,A}^{\alpha,q}(\rr^d,\mathbf{W},
\widetilde{\widetilde{\varphi}})}
\thicksim
\lf\|\vec{f}\,\r\|_{\dot{B}_{p,A}^{\alpha,q}(\rr^d,\mathbf{W},\varphi)}.$$
Therefore, we have
$$\lf\|\lf\{\lf\langle \vec{f},\widetilde{\varphi}_Q\r\rangle\r\}_{Q\in\mathcal{Q}}\r\|_
{\dot{b}_{p,A}^{\alpha,q}(\mathbf{W})}
=\lf\|\lf\{\vec{s}_Q\r\}_{Q\in\mathcal{Q}}\r\|_
{\dot{b}_{p,A}^{\alpha,q}(\mathbf{W})}
\lesssim\lf\|\vec{f}\,\r\|_{\dot{B}_{p,A}^{\alpha,q}(\rr^d,\mathbf{W},\varphi)}.$$
This completes our proof.

\end{proof}

\begin{remark}{\rm
Notice that, if $(\varphi',\psi')$, $(\varphi'',\psi'')$ satisfy the condition \eqref{test2},
$\alpha\in\rr$, $q\in(0,\fz]$, $p\in[1,\fz)$ and $\mathbf{W}\in A_p$, then
$$\dot{B}_{p,A}^{\alpha,q}(\rr^d,\mathbf{W},\varphi')
=\dot{B}_{p,A}^{\alpha,q}(\rr^d,\mathbf{W},\varphi'').$$
Indeed,
Let $(\varphi',\psi')$, $(\varphi'',\psi'')$ satisfy \eqref{test2}. We can decompose $\vec{f}$
as the following:
$$\vec{f}=\sum_{Q\in\mathcal{Q}}\lf\langle \vec{f},\varphi''_Q\r\rangle\psi_Q''
=\sum_{Q\in\mathcal{Q}}\vec{s}_Q''\psi_Q''.$$
From the fact that, for any $Q\in\mathcal{Q}$,  $\psi_Q''$ is a molecule, Lemma \ref{mole} and Theorem \ref{t3.3}, we conclude that
$$\lf\|\vec{f}\,\r\|_{\dot{B}_{p,A}^{\alpha,q}(\rr^d,\mathbf{W},\varphi')}
\lesssim\lf\|\lf\{\vec{s}_Q''\r\}_{Q\in\mathcal{Q}}\r\|_{\dot{B}_{p,A}^
{\alpha,q}(\rr^d,\mathbf{W})}
\lesssim\lf\|\vec{f}\,\r\|_{\dot{B}_{p,A}^{\alpha,q}(\rr^d,\mathbf{W},\varphi'')}.$$
Reversing the role of $\varphi'$ and $\varphi''$, we obtain the equivalence of $\dot{B}_{p,A}^{\alpha,q}(\rr^d,\mathbf{W},\varphi')$ and $\dot{B}_{p,A}^{\alpha,q}(\rr^d,\mathbf{W},\varphi'')$.
Therefore, we know that the definition of $\dot{B}_{p,A}^{\alpha,q}(\rr^d,\mathbf{W})$ is independent of the choice of the test function $\varphi$.
}
\end{remark}

\subsection{The relationship with reducing operators}

In this subsection, we will establish the relation of Besov spaces and  reducing operators.

Recall that, for any matrix weight $\mathbf{W}$, we know that there exist a sequence of
reducing operators $\{\mathcal{A}_Q\}_{Q\in\mathcal{Q}}$ such that, for any $\vec{u}\in\mathcal{H}$,
$$\omega_{p,Q}=\lf[\fint_Q\lf\|\mathbf{W}^{\frac1p}(x)\vec{u}\r\|_{\mathcal{H}}\mathbf{1}_Q(x)\,dx\r]
^{\frac1p}.$$

\begin{lemma}
Let $\{\mathcal{A}_Q\}_{Q\in\mathcal{Q}}$ be the reducing operators for $\mathbf{W}$, $\alpha\in\rr$, $q\in(0,\fz]$, $p\in[1,\fz)$.
Then  we have
$$\lf\|\{\vec{s}_Q\}_Q\r\|_{\dot{b}_{p,A}^{\alpha,q}(\rr^d,\mathbf{W})}=
\lf\|\{\vec{s}_Q\}_Q\r\|_{\dot{b}_{p,A}^{\alpha,q}(\{\mathcal{A}_Q\})}.$$

\end{lemma}

\begin{proof}
By the definition, we have

\begin{align*}
\lf\|\{\vec{s}_Q\}_Q\r\|_{\dot{b}_{p,A}^{\alpha,q}(\rr^d,\mathbf{W})}
=&\lf[\sum_{k\in\zz}|\det A|^{kq\alpha}\lf\|\sum_{Q\in\mathcal{Q},|Q|=|\det A|^{-k}}|Q|^{-\frac 12}\lf\|\mathbf{W}^{\frac1p}(x)\vec{s}_Q\r\|\mathbf{1}_Q(x)\r\|_{L^p(dx)}^q\r]^{\frac1q}\\
=&\lf[\sum_{k\in\zz}|\det A|^{kq\alpha}\lf[\sum_{Q\in\mathcal{Q},|Q|=|\det A|^{-k}}|Q|^{1-\frac p2}\lf(\omega_{p,Q}(\vec{s}_Q)\r)^p\r]^{\frac qp}\r]^{1q}\\
\thicksim&\lf[\sum_{k\in\zz}|\det A|^{kq\alpha}\lf(\sum_{Q\in\mathcal{Q},|Q|=|\det A|^{-k}}|Q|^{-\frac p2}\lf\|\mathcal{A}_{Q}\vec{s}_Q\r\|_{\mathcal{H}}^p\int_Q\mathbf{1}(x)\,dx\r)^{\frac qp}\r]^{1q}\\
\thicksim&\lf[\sum_{k\in\zz}|\det A|^{kq\alpha}\lf\|\sum_{Q\in\mathcal{Q},|Q|=|\det A|^{-k}}|Q|^{-\frac 12}\lf\|\mathcal{A}_{Q}\vec{s}_Q\r\|\mathbf{1}_Q(x)\r\|_{L^p(dx)}^q\r]^{\frac1q}\\
\thicksim&
\lf\|\{\vec{s}_Q\}_Q\r\|_{\dot{b}_{p,A}^{\alpha,q}(\{\mathcal{A}_Q\})}.
\end{align*}
\end{proof}

\begin{thm}
Let $\alpha\in\rr$, $q\in(0,\fz]$, $p\in[1,\fz)$ and $\mathbf{W}\in A_p$. Then the Besov space
$\dot{B}_{p,A}^{\alpha,q}(\rr^d,\mathbf{W})$ is complete.

\end{thm}
\begin{proof}
Let $\{\vec{f}_i\}_{i\in\nn}$ be a Cauchy sequence in $\dot{B}_{p,A}^{\alpha,q}(\rr^d,\mathbf{W})$. From the above lemma, we know that
$\{\{\vec{s}_Q(\vec{f}_i)\}_Q\}_{i\in\nn}$ be a Cauchy sequence in $\dot{b}_{p,A}^{\alpha,q}(\{\mathcal{A}_Q\})$. Therefore, we have
\begin{align*}
&\lf\|\sum_{Q\in\mathcal{Q},|Q|=|\det A|^{-k}}|Q|^{-\frac 12}\lf\|\mathcal{A}_{Q}\lf(\vec{s}_Q(\vec{f}_i)-\vec{s}_Q(\vec{f}_l)\r)\r\|_{\mathcal{H}}\mathbf{1}_Q(x)
\r\|_{L^p(dx)}^p\\
=&|\det A|^{(\frac p2-1)k}\sum_{Q\in\mathcal{Q},|Q|=|\det A|^{-k}}\lf\|\mathcal{A}_{Q}\lf(\vec{s}_Q(\vec{f}_i)-\vec{s}_Q(\vec{f}_l)\r)\r\|
_{\mathcal{H}}^p\\
&\rightarrow 0, \ \ \mathrm{as} \ \ i,l\rightarrow\fz,
\end{align*}
which implies that, for any $Q\in\mathcal{Q}$,
$$\lf\|\mathcal{A}_{Q}\lf(\vec{s}_Q(\vec{f}_i)-\vec{s}_Q(\vec{f}_l)\r)\r\|
_{\mathcal{H}}^p
\rightarrow 0, \ \ \mathrm{as} \ \ i,l\rightarrow\fz.$$
By the fact that $\mathcal{A}_{Q}$ is invertible, we obtain that, for any $Q\in\mathcal{Q}$,  $\{\vec{s}_Q(\vec{f}_i)\}_{i\in\nn}$ is a vector-valued Cauchy sequence in $\mathcal{H}$.
Furthermore, we define  $\vec{s}_Q=\lim_{i\rightarrow\fz}\vec{s}_Q(\vec{f}_i)$.
Let
$$\vec{f}=\sum_{q\in\mathcal{Q}}\vec{s}_Q\psi_Q.$$ Then we have
\begin{align*}
\lf\|\vec{f}_i-\vec{f}\,\r\|_{\dot{B}_{p,A}^{\alpha,q}(\rr^d,\mathbf{W})}
=&\lf\|\sum_{Q\in\mathcal{Q}}\lf(\vec{s}_Q(\vec{f}_i)-\vec{s}_Q\r)\psi_Q\r\|
_{\dot{B}_{p,A}^{\alpha,q}(\rr^d,\mathbf{W})}\\
\lesssim&\lf\|\lf\{\vec{s}_Q(\vec{f}_i)-\vec{s}_Q\r\}_{Q\in\mathcal{Q}}\r\|
_{\dot{b}_{p,A}^{\alpha,q}(\mathbf{W})}\\
\lesssim&\liminf_{l\rightarrow\fz}\lf\|\lf\{\vec{s}_Q(\vec{f}_i)-
\vec{s}_Q(\vec{f}_l)\r\}_{Q\in\mathcal{Q}}\r\|
_{\dot{b}_{p,A}^{\alpha,q}(\{\mathcal{A}_Q\})}\rightarrow0,\ \  \mathrm{as} \ \ l\rightarrow\fz.
\end{align*}
From this, the Fatou Lemma and the fact that $\{\{\vec{s}_Q(\vec{f}_i)\}_{Q\in\mathcal{Q}}\}_{i\in\nn}$ is a Cauchy sequence in
$\dot{b}_{p,A}^{\alpha,q}(\{\mathcal{A}_Q\})$, we deduce that
$$\vec{f}=(\vec{f}+\vec{f}_i)-\vec{f}_i\in\dot{B}_{p,A}^{\alpha,q}(\rr^d,\mathbf{W}).$$
This implies that the Besov space $\dot{B}_{p,A}^{\alpha,q}(\rr^d,\mathbf{W})$ is complete.

\end{proof}

\begin{corollary}
Let $\alpha\in\rr$, $q\in(0,\fz]$, $p\in(1,\fz)$, and $\mathbf{W}$ be a matrix weight  with corresponding reducing operators $\mathcal{A}_Q$ and $\widetilde{\mathcal{A}}_Q$. Then we have
\begin{enumerate}
\item[\rm{(i)}]
$\dot{b}_{p,A}^{\alpha,q}(\{\mathcal{A}_Q\})\subseteq
\dot{b}_{p,A}^{\alpha,q}(\{(\widetilde{\mathcal{A}}_Q)^{-1}\})$;

\item[\rm{(ii)}] If $\omega\in A_p$,
$\dot{b}_{p,A}^{\alpha,q}(\{(\widetilde{\mathcal{A}}_Q)^{-1}\})
\subseteq\dot{b}_{p,A}^{\alpha,q}(\{\mathcal{A}_Q\}).$
\end{enumerate}

\end{corollary}

\section{Inhomogeneous matrix-weighted Besov spaces}\label{s4}

In this section, we will investgate the inhomogeneous Besov spaces. Before we define the vector-valued inhomogeneous Besov space $B^{\alpha,q}_{p,A}(\rr^d,\mathbf{W})$ with matrix weight $\mathbf{W}$, and briefly give some basic properties of $B^{\alpha,q}_{p,A}(\rr^d,\mathbf{W})$.
Most of them can be deduced by a straightforward modification of the homogeneous results, therefore, we only outline require changes.

Let $\Phi,\varphi\in\cs(\rr^d)$ and satisfy
\begin{align*}
\supp \hat{\Phi}\subset[-\pi,\pi]^d
\end{align*}
and
\begin{align*}
\sup_{k\geq1} \lf|\hat{\varphi}((A^*)^{-k}\xi)\r|>0, \ |\hat{\Phi}(\xi)|>0 \ \ \mathrm{for \ any} \ \xi\in\rr^d.
\end{align*}

In what follows, for $\varphi\in \cs(\rn)$, $k\in\zz$ with $k>0$, and $x\in\rn$, let $$\varphi_k(x):= |\det A|^{-k}\varphi(A^{-k}x).$$

\begin{definition}
Let $\alpha\in\rr$, $p,q\in(0,\fz]$. We define the matrix-weighted anisotropic inhomogeneous Besov space $B^{\alpha,q}_{p,A}(\rr^d,\mathbf{W})$ as the set of all $\vec{f}\in\cs'$ such that
$$\lf\|\vec{f}\,\r\|_{B^{\alpha,q}_{p,A}(\rr^d,\mathbf{W})}:=
\lf\|\vec{f}*\Phi\r\|_{L^p(\mathbf{W})}+
\lf(\sum_{k=1}^{\fz}|\det A|^{\alpha qk}\lf\|\vec{f}*\varphi_k\r\|_{L^p(\mathbf{W})}^q\r)^{\frac1q}<\fz.$$
\end{definition}

The corresponding inhomogeneous weighted sequence Besov space $b^{\alpha,q}_{p,A}(\mathbf{W})$ is
defined for the vector sequences enumerated by the dyadic cubes $Q$ with $|Q|\leq1$.
\begin{definition}
Let $\alpha\in\rr$, $p,q\in(0,\fz]$ and $\mathbf{W}$ be a matrix weight. We define the matrix-weighted sequence anisotropic inhomogeneous Besov space $b^{\alpha,q}_{p,A}(\mathbf{W})$ as the set of all all vector-valued sequences $\vec{s}=\{\vec{s}_Q\}_{Q\in\mathcal{Q},l(Q)\leq1}$ such that
$$\lf\|\vec{f}\,\r\|_{b^{\alpha,q}_{p,A}(\mathbf{W})}:=
\lf(\sum_{k=1}^{\fz}|\det A|^{\alpha qk}\lf\|\sum_{|Q|=|\det A|^{-k}}|Q|^{-\frac12}\vec{s}_Q\mathbf{1}_Q\r\|_{L^p(\mathbf{W})}^q\r)^{\frac1q}<\fz.$$
When $q=\fz,$ the $\ell^q$-norm can be replaced by the supremum on $k\geq0$.

\end{definition}

Next we also introduce the definitions of $\varphi$-transform $S_{\varphi}$ and the inverse $\varphi$-transform $T_{\varphi}$ in the inhomogeneous setting.

Let $\Phi, \Psi\in\cs$, $\varphi, \psi\in\cs$ and satisfy
\begin{align*}
\supp \hat{\varphi}, \supp \hat{\psi}\subset[-\pi,\pi]^d\backslash\{0\}
\end{align*}
and
\begin{align}\label{test2}
\overline{\hat{\Phi}}(\xi)\hat{\Psi}(\xi)+\sum_{k=1}^{\fz}\overline{\hat{\varphi}((A^*)^{-k}\xi)}
\hat{\psi}((A^*)^{-k}\xi)=1,  \ \ \mathrm{for \ any} \ \xi\in\rr^d.
\end{align}
For any $\mathbf{j}\in\zz^d$, $k\in\zz$ with $k>0$, define
$$Q_{\mathbf{j},k}=A^{-k}([0,1]^d+\mathbf{j})$$
be the dilated cube, $x_{Q_{\mathbf{j},k}}=A^{-k}\mathbf{j}$.
$$\mathcal{Q}=\lf\{Q_{\mathbf{j},k}:k\in\zz, \mathbf{j}\in\zz^d\r\}.$$
Define

$$\Phi_Q=|\det A|^{-\frac{k}{2}}\Phi(A^{k}x-\mathbf{j})=
|Q|^{\frac12}\Phi(x-x_Q)$$
and
$$\varphi_Q=|\det A|^{-\frac{k}{2}}\varphi(A^{k}x-\mathbf{j})=
|Q|^{\frac12}\varphi(x-x_Q).$$

\begin{definition}
Let $\Phi, \Psi\in\cs$, and  $\varphi, \psi\in\cs$ satisfy \eqref{test2}.
For any vector-valued distribute $\vec{f}$, $f_i\in\cs'$, define the inhomogeneous
$\varphi$-transform $S_{\varphi}=S_{\Phi,\varphi}$:
$S_{\varphi}(\vec{f}):=\{S_{\varphi}(\vec{f})\}_{Q},$
if $k>0$, $$\{S_{\varphi}(\vec{f})\}_{Q}=\lf\{\langle\vec{f},\varphi_Q\rangle\r\}_Q
=\lf\{\lf(\langle f_1,\varphi_Q\rangle,\langle f_2,\varphi_Q\rangle,\ldots,\langle f_m,\varphi_Q\rangle\r)^T\r\}_Q;$$
if $k=0$, $$\{S_{\varphi}(\vec{f})\}_{Q}=\lf\{\langle\vec{f},\Phi_Q\rangle\r\}_Q
=\lf\{\lf(\langle f_1,\Phi_Q\rangle,\langle f_2,\Phi_Q\rangle,\ldots,\langle f_m,\Phi_Q\rangle\r)^T\r\}_Q;$$

For any vector-valued sequence $\vec{s}=\{s_Q\}_{Q,|Q|\leq0}$, define the inhomogeneous inverse
$\varphi$-transform $T_{\psi}=T_{\Psi,\psi}$:
$$T_{\psi}(\vec{s}):=\sum_{\mathcal{Q},|Q|=1}\vec{s}_Q\Psi_Q
+\sum_{Q\in\mathcal{Q},|Q|<1}\vec{s}_Q\Psi_Q,$$

\end{definition}

Given $\Phi, \varphi\in\cs$ satisfying \eqref{test2}, we can show that there exist $\Psi,\psi\in\cs$ such that
\begin{align}\label{test4}
\overline{\hat{\Phi}(\xi)}
\hat{\Psi}(\xi)+\sum_{k=1}^{\fz}\overline{\hat{\varphi}((A^*)^{-k}\xi)}
\hat{\psi}((A^*)^{-k}\xi)=1,  \ \ \mathrm{for \ any} \ \xi\in\rr^d.
\end{align}
Moreover, we have the following identity for vector-valued distribution $\vec{f}$, $f_i\in\cs',$

$$\vec{f}=\sum_{Q\in\mathcal{Q},|Q|=1}\lf\langle \vec{f},\Phi_Q\r\rangle\Psi_Q
+\sum_{k=1}^{\fz}\sum_{Q\in\mathcal{Q},|Q|=|\det A|^{-k}}\lf\langle \vec{f},\varphi_Q\r\rangle\Psi_Q,$$
with convergence in $\cs'$.

We can then show that main results of $\varphi$-transform adapt directly to inhomogeneous setting.
That is, $S_{\varphi}$ is a bounded operator from $B^{\alpha,q}_{p,A}(\rr^d,\mathbf{W})$ to
$b^{\alpha,q}_{p,A}(\mathbf{W})$ and $T_{\psi}$ is a bounded operator
from $b^{\alpha,q}_{p,A}(\mathbf{W})$ to $B^{\alpha,q}_{p,A}(\rr^d,\mathbf{W})$.

\section{Almost diagonal operators}

In this section, we study a class of almost diagonal operators on Besov spaces, which
was introduced in the dyadic case by Frazier et al. \cite{fj89}.  The interest of these operators comes from their close
connection to operators on Besov spaces.

\begin{definition}
Let $\alpha\in\rr$, $p,q\in(0,\fz]$.
Suppose that $J=\beta/p+\max\{0,1-1/p\}$. We say
that a matrix $\{a_{Q,P}\}_{Q,P\in\mathcal{Q}}$ is almost diagonal, 
$A\in\mathbf{ad}^{\alpha,q}_{p,A}(\beta),$ if there exists an $c>0$ such that,
$$\sup_{Q,P\in\mathcal{Q}}
\lf|a_{Q,P}\r|<c
\lf(\frac{|Q|}{|P|}\r)^{\alpha}\lf[1+\frac{\rho_A(x_Q-x_P)}{\max\{|P|,\,|Q|\}}\r]^{-J-c}
\min\lf\{\lf(\frac{|Q|}{|P|}\r)^{\frac{1+c}{2}},\,\lf(\frac{|P|}{|Q|}\r)^{\frac{1+c}{2}+J-1}\r\}.
$$
\end{definition}

The following conclusion demonstrate that we study some property of operators on Besov space $\dot{B}^{\alpha,q}_{p,A}(\rr^d,\mathbf{W})$ by considering corresponding operators on $\dot{b}^{\alpha,q}_{p,A}(\mathbf{W})$.
 \begin{thm}\label{t8.1}
 Let $\alpha\in\rr$, $0<q\leq\fz$, $1\leq p<\fz$, and let $\mathbf{W}$ be a doubling
matrix of order $p$ with doubling exponent $\beta$. If $A\in \mathbf{ad}^{\alpha,q}_{p,A}(\beta)$, 
then $A:\dot{b}^{\alpha,q}_{p,A}(\mathbf{W})\rightarrow \dot{b}^{\alpha,q}_{p,A}(\mathbf{W})$ is bounded.
\end{thm}

\begin{definition} 
Let $T$ be a continuous linear operator from $\cs'$ to $\cs'$. We say that
$T$ is an almost diagonal operator for 
$B^{\alpha,q}_{p,A}(\rr^d,\mathbf{W})$, and write $T\in\mathbf{AD}^{\alpha,q}_{p,A}(\beta)$, if for some
pair of mutually admissible kernels $(\varphi,\psi)$, the matrix $\{a_{Q,P}\}\in\mathbf{ad}^{\alpha,q}_{p,A}(\beta)$, where
$a_{Q,P}=\lf\langle T \psi_P, \varphi_Q\r\rangle.$
\end{definition}
\begin{remark} {\rm It is easy to see that,
the definition of $T\in\mathbf{AD}^{\alpha,q}_{p,A}(\beta)$ is independent of the choice of the
pair $(\varphi,\psi)$.}
\end{remark}

As a straightforward consequence of Theorem \ref{t8.1}, we can deduce the following conclusion:
\begin{corollary}
 Let $T\in\mathbf{AD}^{\alpha,q}_{p,A}(\beta)$, $\alpha\in\rr$, $p\in[1,\fz)$ and $q\in(0,\fz)$. Then $T$ extends
to a bounded operator on $\dot{B}_{p,A}^{\alpha,q}(\rr^d,\mathbf{W})$ if $\mathbf{W}\in A_p.$
\end{corollary}

\begin{proof}
By the density, we need to consider $\vec{f}$ with $f_i\in\cs$, $i=1,2,\ldots,m$. 
Let $(\varphi,\psi)$ be a pair of mutually admissible
kernels. 
Let 
$$\vec{t}_Q=\sum_{P\in\mathcal{Q}} \langle T\psi_P,\, \varphi_Q\rangle \vec{s}_P(\vec{f}).$$ 
Notice that $\{\langle T\psi_P, \varphi_Q\rangle\}_{Q,P}\in \mathbf{ad}^{\alpha,q}_{p,A}(\beta)$. From the $\varphi$-transform decomposition
$$\vec{f}=\sum_{P\in\mathcal{Q}}\vec{s}_P(\vec{f})\psi_P,$$ 
and Theorem \ref{t8.1}, 
we can deduce that
\begin{align*}
\lf\|T(\vec{f})\r\|_{\dot{B}_{p,A}^{\alpha,q}(\rr^d,\mathbf{W})}
=&\lf\|\sum_{P\in\mathcal{Q}}\vec{s}_P(\vec{f})T\psi_P\r\|_{\dot{B}_{p,A}^{\alpha,q}
(\rr^d,\mathbf{W})}\\      
=&\lf\|\sum_{Q\in\mathcal{Q}}\lf(\sum_{P\in\mathcal{Q}}\lf\langle T\psi_P, \varphi_Q\r\rangle\vec{s}_P(\vec{f})\r)\psi_Q\r\|_{\dot{B}_{p,A}^{\alpha,q}(\rr^d,\mathbf{W})}\\
=&\lf\|\sum_{Q\in\mathcal{Q}}\vec{t}_Q\psi_Q\r\|_{\dot{B}_{p,A}^{\alpha,q}(\rr^d,\mathbf{W})}
\lesssim\lf\|\lf\{\vec{t}_Q\r\}_{Q\in\mathcal{Q}}\r\|_{\dot{b}^{\alpha,q}_{p,A}(\mathbf{W})}\\
\lesssim&\lf\|\lf\{\vec{s}_Q\r\}_{Q\in\mathcal{Q}}\r\|_{\dot{b}^{\alpha,q}_{p,A}(\mathbf{W})}
\lesssim\lf\|\vec{f}\,\r\|_{\dot{B}_{p,A}^{\alpha,q}(\mathbf{W})}.
\end{align*}

\end{proof}

\textbf{Acknowledgements.}
X. Liu is supported by the Gansu Province Education Science and Technology Innovation Project (No. 224040), and
the Foundation for Innovative Fundamental Research Group Project of Gansu Province (No. 25JRRA805).
W. Wang is supported by
 China Postdoctoral Science Foundation (No. 2024M754159), and Postdoctoral Fellowship Program of CPSF (No. GZB20230961).


\bigskip

\end{document}